\documentclass{amsart}

\usepackage{articletemplate}
\usepackage{xcolor}

\newcommand{\vs}{\vspace{3mm}}

\newcommand{\N}{\mathbb{N}}

\title{On the Length of Pierce Expansions}
\author{Zachary Chase \and Mayank Pandey}
\thanks{The first author is partially supported by Ben Green's Simons Investigator Grant 376201 and gratefully acknowledges the support of the Simons Foundation.}
\address{Mathematical Institute, Andrew Wiles Building, Radcliffe Observatory Quarter, Woodstock Road, Oxford OX2 6GG, UK}
\email{zachary.chase@maths.ox.ac.uk}
\address{Department of Mathematics, Princeton University, Princeton, NJ 08540, USA}
\email{mayankpandey9973@gmail.com}
\date{October 5, 2022}

\begin{document}

\begin{abstract}
For a given positive integer $n$, how long can the process $x \mapsto n \pmod{x}$ last before reaching $0$? We improve Erd\H{o}s and Shallit's upper bound of $O(n^{\frac{1}{3}+\eps})$ to $O(n^{\frac{1}{3}-\frac{2}{177}+\eps})$ for any $\eps > 0$.
\end{abstract}

\maketitle

\section{Introduction}

The continued fraction expansion of a real number $x \in (0,1)$, given by $$x = \frac{1}{a_1+\frac{1}{a_2+\dots}},$$ plays an important role throughout number theory. The terms $a_i$ can be extracted, for example, from the iterated process $t \mapsto \frac{1}{t} \pmod{1}$ beginning with $t=x$. It is well-known and not difficult to see that the continued fraction expansion of a real number $x$ is finite if and only if $x$ is a rational number. And if $x$ is rational, the sequence of terms $a_i$ produced are exactly the quotients produced by the classic Euclidean algorithm applied to the numerator and denominator.    

\vs

In this paper, we are concerned with the \textit{Pierce expansion} of a real number $x \in (0,1)$, introduced by Pierce \cite{pierce} and named by Shallit \cite{shallit}. Here, the expansion is of the form $$x = \frac{1}{b_1}-\frac{1}{b_1b_2}+\frac{1}{b_1b_2b_3}-\dots,$$ where now the terms $b_i$ can be extracted from the iterated process $t \mapsto 1 \pmod{t}$ beginning with $t=x$. It is also not difficult to see that the Pierce expansion of a real number $x$ is finite if and only if $x$ is rational (see, e.g., \cite{shallit}). And if $x$ is rational, the sequence of terms $b_i$ produced are exactly the quotients produced by an algorithm that at first glance appears similar to Euclid's algorithm.

\vs

Let us give an example of the algorithm. Say $x = \frac{13}{35}$. We start with $13$ and repeatedly obtain successive integers by reducing $35$ modulo the current number. For example, 
$$35 = 2\cdot 13 + 9$$ $$\hspace{-2mm} 35 = 3\cdot 9 + 8$$ $$\hspace{-2mm} 35 = 4\cdot 8 + 3$$ $$35 = 11\cdot 3 + 2$$ $$35 = 17\cdot 2 + 1$$ $$35 = 35\cdot 1 + 0$$ gives rise to $$\frac{13}{35} = \frac{1}{2}-\frac{1}{2\cdot 3}+\frac{1}{2\cdot3\cdot4}-\frac{1}{2\cdot3\cdot4\cdot11}+\frac{1}{2\cdot3\cdot4\cdot11\cdot17}-\frac{1}{2\cdot3\cdot4\cdot11\cdot17\cdot 35}.$$ 

\vs

Motivated by the known fact that the Euclidean algorithm used to divide a positive integer $a$ by a positive integer $n$ terminates after $O(\log n)$ steps (which is sharp), it is natural to ask how quickly the above algorithm must terminate 
for a given denominator, no matter the numerator.

\vs

To this end, for positive integers $a,n \in \N$, define $P(a,n)$ to be the first positive integer $k$ such that $a_k = 0$, where $a_0 := a$ and $a_{j+1} = n \pmod{a_j} \in \{0,1,\dots,a_j-1\}$  for $j \ge 0$. In the above example we have $P(a,n) = P(13,35) = 6$. Since we only concern ourselves with the ``length" of the algorithm, we need not keep track of quotients and may compress for instance the above example to $$\hspace{2mm} 35 \pmod{13} = 9 \vspace{-2mm} $$ $$35 \pmod{9} = 8 \vspace{-1.5mm}$$ $$35 \pmod{8} = 3 \vspace{-1.5mm}$$ $$35 \pmod{3} = 2 \vspace{-1.5mm}$$ $$35 \pmod{2} = 1 \vspace{-1.5mm}$$ $$ \hspace{1mm} 35 \pmod{1} = 0.$$  

\vspace{1.5mm}

\noindent Noting $P(a,n) = 2$ if $a > n$, we set $$P(n) := \max_{1 \le a \le n} P(a,n).$$ 
The problem we consider that of obtaining bounds on $P(n)$. Shallit \cite{shallit} proved, using purely ``Archimedean" arguments, that 
$P(n)\ll n^{\frac{1}{2}}$ (see \S\ref{sec:notation} for 
our conventions regarding Vinogradov notation), while 
\[
\limsup_{n\to\infty} \frac{P(n)}{\log n/\log\log n} > 0.
\] 
The upper bound was improved by Erd\H{o}s and Shallit \cite{erdosshallit} who  leveraged
``arithmetic" arguments to combine with the previous ``Archimedean" ones. 
They established $P(n)\ll n^{\frac{1}{3}+\eps}$ and also
improved the lower bound to $\limsup_{n\to\infty} P(n)/\log n > 0$. 
These bounds have since remained the state of the art, with the exponent $1/3$ representing a natural barrier. 


\vs

In this paper, we improve the upper bound on $P(n)$, (slightly) pushing past the $1/3$ barrier. 

\vspace{1.5mm}

\begin{theorem}\label{main}
We have 
\[
    P(n)\ll n^{\frac{1}{3}-\frac{2}{177}+\eps}.
\]
\end{theorem}

\vspace{1.5mm}

We did not put substantial effort into optimizing the exponent gain achieved in Theorem \ref{main}; we could 
not, however, see a way to improve the upper bound to $P(n) \ll n^\eps$ using our techniques.

\vs

Secondly, we establish a lower bound that applies to all $n \in \N$. 
As we can tell, the best bound known prior was $P(n) \gg \log\log n$. 

\vspace{1.5mm}

\begin{theorem}\label{lowerbound}
We have the lower bound
\[
    P(n)\gg\frac{\log n}{\log\log n}
\] 
for all sufficiently large $n$.
\end{theorem}

\vspace{1.5mm}

As one can see, there is an exponential gap between the best known lower and upper bounds on $P(n)$. We hope this paper will reignite interest in determining the true asymptotics and related questions. 

\vs

In \S\ref{sec:notation}, we specify the notational conventions 
we use throughout the paper. In \S\ref{sec:pf_main}, we give the proof of our main theorem, Theorem \ref{main}. In \S\ref{sec:pf_lower}, we give the proof of the lower bound, Theorem \ref{lowerbound}. 

\vspace{1.5mm}

\section{Notation}
\label{sec:notation}

Any statement involving $\eps$ should be read to mean that the statement holds for all $\eps > 0$
We use the standard Vinogradov notation, in which
we write $A \ll B$ (and equivalently $B \gg A$) to denote that
$|A| \le C B$ for some implied constant $C > 0$ that depends only on $\eps$ (if $A, B$ depend on it). 
We write $A \asymp B$ to denote that both $A \ll B$ and $B \ll A$ hold. 
For a parameter $\beta$, we write $\ll_\beta$ and $\asymp_\beta$ to mean that the implied constant may depend on $\beta$.
For positive integers $a,A \in \N$, we write $a \sim A$ to denote $A < a \le 2A$. Finally, we use the standard $e(t) := e^{2\pi i t}$. 

\vspace{1.5mm}

\section{Proof of Theorem \ref{main}}
\label{sec:pf_main}

In this section, we prove our main theorem, that 
$P(n)\ll n^{\frac{1}{3}-\frac{2}{177}+\eps}$. 
We do this by establishing bounds for the amount of time the algorithm
spends in dyadic intervals.

\vs

For the rest of this section, fix a (large) positive integer $n$ and a positive integer $a_0$, letting $a_{j+1} = n \pmod{a_j}$ for $j \ge 0$.

\vs

Write 
\[
T(A) := \#\set*{j \ge 0 : a_j\sim A}.
\]

\vs

The first bound we present on $T(A)$ was proven in \cite{shallit} and is due to ``Archimedean" reasons
(namely that the $a_j$ drop quickly near $n$).

\vspace{1.5mm}

\begin{lemma}\label{archbound}
We have $T(A)\le\frac{n}{2A} + 2$.
\end{lemma}

\begin{proof}
For $j \ge 0$, let $b_j =  \lfloor \frac{n}{a_j} \rfloor$, so that $\frac{n}{b_j+1} < a_j \le \frac{n}{b_j}$. 
We claim that $b_{j+1} > b_j$ for each $j \ge 0$. 
Indeed, if not, $n = b_ja_j+a_{j+1}$, so $a_{j+1} > \frac{n}{b_j+1}$ implies 
$n(b_j+1)-b_ja_j(b_j+1) > n$, which yields $a_j < \frac{n}{b_j+1}$, a contradiction.
Therefore, since $a_j\sim A$ implies $b_j \in [\frac{n}{2A}-1,\frac{n}{A})$, the 
desired bound follows.
\end{proof}

\vs

Note that Lemma \ref{archbound} combined with the trivial $T(A) \le A$ already establishes the bound 
$P(n)\ll n^{1/2}$ of Shallit \cite{shallit}. The second bound we present improves this trivial bound, by taking advantage of ``arithmetic" properties of the iterative process. It was proven in \cite{erdosshallit}.
We reproduce this proof in our own notation as many of its features 
make their way into the proof our improvement.
\vs

\begin{lemma}\label{arithbound1}
For $1\le A\le n$, we have the bound
\[
    T(A) \ll A^{\frac{1}{2}}n^\eps.
\]
\end{lemma}

\begin{proof}
If $T(A)\le 1$, we are done, so suppose that $T(A)\ge 2$.
Let 
\[
    \mc J := \set{j\ge 0 : a_j\sim A, a_{j} - a_{j + 1}\le\frac{4A}{T(A)}}.
\]
Note that 
\[
    \sum_{\substack{j\ge 0\\ a_j, a_{j + 1}\sim A}} 1 = T(A) - 1\ge\frac{1}{2}T(A), 
    \sum_{\substack{j\ge 0\\ a_j, a_{j + 1}\sim A}} (a_j - a_{j + 1})\le A.
\]
It follows that 
\[
    \#\set{j\ge 0 : a_j\sim A, a_j - a_{j + 1} > \frac{4A}{T(A)}} < \frac{1}{4}T(A),
\]
so $\#\mc J\ge\frac{1}{4}T(A)$.
Now, note that for all $j$, 
\[
    a_{j + 1}\equiv n \pmod{a_j}\implies a_j | n + a_{j} - a_{j + 1}.
\]
We obtain that 
\[
    T(A)\ll\#\mc J\le
    \sum_{h\le\frac{4A}{T(A)}}\sum_{\substack{a\sim A\\ a | n + h}} 1.
\]
By the divisor bound, $\sum_{\substack{a\sim A\\ a | n + h}} 1\le d(n + h)\ll n^\eps$, so 
we obtain
\[
    T(A)\ll\frac{A}{T(A)}n^\eps.
\]
Rearranging yields the desired result.
\end{proof}

\vs

Together, Lemmas \ref{archbound}, \ref{arithbound1} applied to the ranges
$A\ge n^{2/3}, A\le n^{2/3}$, respectively, give the bound $P(n) \ll n^{\frac{1}{3}+\eps}$. 
To obtain a bound of $n^{\frac{1}{3} - \delta + \eps}$, it 
suffices to show that $T(A)\ll n^{\frac{1}{3} - \delta + \eps}$ for 
$A\in [n^{\frac{2}{3} - 2\delta}, n^{\frac{2}{3} + \delta}]$. This is 
the content of Proposition \ref{arithbound2} for sufficiently small $\delta > 0$. 
To do this, we make use of the arithmetic information obtained by analyzing two consecutive jumps. After using Poisson summation, we are reduced, roughly, to obtaining a power saving over
the trivial bound for the sum 
\[
    \sum_{b\sim n^{1/3}}e\pfrc{n}{b}.
\]
Such bounds follow from standard exponential sum bounds. 
In our case, we use the exponent pair $\prn{\frac{13}{84} + \eps,\frac{55}{84} + \eps}$ 
of Bourgain \cite{B}. Much simpler methods would have also worked, to give a slightly worse
saving over the trivial bound (the van der Corput A-process, followed by the B-process, 
for example). 


\vs

\begin{proposition}\label{arithbound2}

Suppose that $\delta, \lambda > 0$ are such that 
\[
    \delta < \frac{1}{18}, \hspace{2mm} \lambda\le\frac{1}{3} - \delta.
\]
Then, for $n^{\frac{2}{3} - 2\delta}\le A\le n^{\frac{2}{3} + \delta}$, we 
have
\[
    T(A)\ll 
    n^{\frac{1}{3} - \gamma + \eps},
\]
where 
\[
    \gamma := \min\bigg(\lambda - 2\delta, \delta, 
    \frac{4}{63} - \frac{349}{84}\delta - \frac{13}{84}\lambda\bigg).
\]

\end{proposition}

\vs

Before proving Proposition \ref{arithbound2}, let us first quickly spell out how Theorem \ref{main} follows. 

\vs

\begin{proof}[Proof of Theorem \ref{main} assuming Proposition \ref{arithbound2}]
Take 
\[
    \delta = \frac{2}{177}, \hspace{2mm} \lambda = 3\delta.
\]
It is easy to check that $\delta, \lambda$ satisfy the hypotheses of Proposition \ref{arithbound2}.
We have that 
\[
    P(n)\le 1 + \sum_{A\le n} T(A),
\]
where the sum over $A$ runs over only powers of $2$. The contribution of 
$A > n^{\frac{2}{3} + \delta}$ is, by Lemma \ref{archbound}, 
\[
    \ll\sum_{n^{\frac{2}{3} + \delta} < A\le n}\frac{n}{A}\ll n^{\frac{1}{3} - \delta}.
\]
By Lemma \ref{arithbound1}, the contribution of $A < n^{\frac{2}{3} - 2\delta}$ is 
\[
    \ll n^\eps\sum_{A < n^{\frac{2}{3} - 2\delta}} A^{\frac{1}{2}}
    \ll n^{\frac{1}{3} - \delta + \eps},
\]
For $n^{\frac{2}{3} - 2\delta}\le A\le n^{\frac{2}{3} + \delta}$, by Proposition
\ref{arithbound2}, we have that 
\[
    T(A)\ll n^{\frac{1}{3} - \gamma + \eps},
\]
where 
\[
    \gamma = \min\bigg(\lambda - 2\delta, \delta, 
    \frac{4}{63} - \frac{349}{84}\delta - \frac{13}{84}\lambda\bigg) = \frac{2}{177}.
\]
Then, summing over $A$ in $[n^{\frac{2}{3} - 2\delta}, n^{\frac{2}{3} + \delta}]$ 
at the harmless cost of $O(\log n)$, Theorem \ref{main} follows.
\end{proof}

\begin{proof}[Proof of Proposition \ref{arithbound2}]
Suppose that $T(A)\ge T_0 = n^{\frac{1}{3} - \delta}$, for we are done otherwise. 
Let $m$ be so that $a_{m + T(A)}\le A < a_{m + T(A) - 1} < \dots < a_m\le 2A$.
Then, as in the proof of Lemma \ref{arithbound1}, for a positive proportion of $m + 2\le j < m + T(A)$, we have that 
\[
    a_{j - 2} - a_{j}\le H := \frac{10A}{T_0}.
\]
We record the bound $n^{\frac{1}{3} - \delta}\ll H\ll n^{\frac{1}{3} + 2\delta}$.
Write 
\[
    \mc J = \set*{m + 2\le j < m + T(A) : a_{j -2} - a_j\le H}.
\]
Consider some $j\in\mc J$, and write 
$a = a_{j - 2}, a - h =  a_{j - 1}, a - h - h' = a_{j}$. 
Then, as in the proof of Lemma \ref{arithbound1}, we have
\[
    a | n + h, a - h | n + h'.
\]
In particular, there exist $b\asymp n/A, k$ such that 
$ab = n + h, (a - h)(b + k) = n + h'$.

Also, note that 
\[
    |(b + k)h - ak| = |ab - (a - h)(b + k)|\ll H,
\]
so rearranging, we have
\[
    h = \frac{ak}{b + k} + O\pfrc{AH}{n}
     = \frac{abk}{b(b + k)} + O\pfrc{AH}{n} = \frac{nk}{b(b + k)} + O\pfrc{AH}{n}
\]
since $Hk/B^2\ll H/B = AH/n$.
Write $H_0(b, k) := \frac{nk}{b(b + k)}$. 
Recall that $\lambda\le\frac{1}{3} - \delta$, so for $b\asymp\frac{n}{A}$
\[
    H_0(b, k)n^{-\lambda}\ge\frac{H_0(b, k)}{T_0}\gg\frac{A^2}{nT_0}\asymp\frac{AH}{n}.
\]
It follows for some sufficiently large $C > 0$ that
$\charf{|h - H_0(b, k)|\ll AH/n}\le\charf{|h - H_0(b, k)|\le L}$ with $L := CH_0(n/A, k)n^{-\lambda}$.
The reason for this apparently wasteful bound is to lower the ``analytic conductor" of the 
phase in the resulting exponential sum so that we may get superior savings when
we execute the sum over $b$. It follows that
\begin{align*}
    \#\mc J
    &\le\sum_{|k|\ll Hn/A^2}\sum_{h\le H}\sum_{\substack{b | n + h\\ b\asymp n/A}}
    \charf{|h - H_0(b, k)|\ll\frac{AH}{n}}\\
    &\le\sum_{|k|\ll Hn/A^2}\sum_{h\le H}\sum_{\substack{b | n + h\\ b\asymp n/A}}
    \charf{|h - H_0(b, k)|\le L}.
\end{align*} 
Take some smooth even $w$ so that $\charf{[-1, 1]}\le w\le\charf{[-2, 2]}$.
Then, we have
\[
    \sum_{|k|\ll Hn/A^2}\sum_{h\le H}\sum_{\substack{b | n + h\\ b\asymp n/A}}
    \charf{|h - H_0(b, k)|\le L}
    \le \sum_{|k|\ll Hn/A^2}\sum_{b\asymp n/A}\sum_{h\equiv -n(b)}
    w\pfrc{h - H_0(b, k)}{L}.
\]
By Poisson summation, 
\begin{align*}
    \sum_{|k|\ll Hn/A^2}&\sum_{b\asymp n/A}\sum_{h\equiv -n(b)}
    w\pfrc{h - H_0(b, k)}{L}\\
    &= \sum_{|k|\ll Hn/A^2}\sum_{b\asymp n/A}\frac{L}{b}
    \sum_{r\in\ZZ}e\pfrc{r(n + H_0)}{b}\hat w\pfrc{Lr}{b}.
\end{align*}
The contribution of the zero frequency, $r=0$, is 
\[
    \ll \frac{Hn}{A^2}\cdot\frac{n}{A}\cdot\frac{Hn^{-\lambda}}{n/A}\ll n^{\frac{1}{3} + 2\delta -\lambda},
\]
which is acceptable. It remains to bound the contribution when $|r| > 0$,
so we restrict to that case from now on.

A quick computation shows that for $x_0\asymp n/A$, we have that for some constant 
\[
    \frac{\text d^j}{\text dx^j}\left. \frac{r(n + H_0(x, k))}{x} \right|_{x = x_0}\asymp_j rAx_0^{-j}
\]
uniformly in $|k| \ll Hn/A^2$.

\vs

By Theorem 6 of \cite{B}, we have the exponent pair 
$(\frac{13}{84} + \eps, \frac{55}{84} + \eps)$
(see \S8.4 of \cite{IK} for a definition; note that in 
the notation of \cite{IK}, we instead have the exponent pair
$(13/84, 13/84)$). By partial summation 
(see, e.g., \cite[Lemma 2.2]{mrt}), the fact that $\hat w, (\hat w)'$ are Schwartz, 
and that $|r| > 0$ (which implies that $|r|A\gg n/A$, so (8.56) of \cite{IK} holds), 
we have for some $c > 0$ that
\begin{align*} 
    \bigg|\sum_{b\asymp n/A}&e\pfrc{r(n + H_0(b, k))}{b}\frac{n/A}{b}
    \hat w\pfrc{Lr}{b}\bigg|\\
    &\ll\bigg(1 + \bigg|\frac{Lr}{n/A}\bigg|\bigg)^{-2022}
    \sup_{t\ll n/A}\bigg|\sum_{cn/A < b\le t}e\pfrc{r(n + H_0(b, k))}{b}\bigg|\\
    &\ll\bigg(1 + \bigg|\frac{Lr}{n/A}\bigg|\bigg)^{-2022}
    \pfrc{A^2|r|}{n}^{\frac{13}{84} + \eps}\pfrc{n}{A}^{\frac{55}{84} + \eps}.
\end{align*}
Putting this all together, we obtain that 
\begin{align*}
    \bigg|&\sum_{|k|\ll Hn/A^2}\sum_{b\asymp n/A}\frac{L}{n/A}
    \sum_{r\ne 0}e\pfrc{r(n + H_0(b, k))}{b}\frac{n/A}{b}\hat w\pfrc{Lr}{b}\bigg|\\ 
    &\ll\frac{Hn}{A^2}\cdot\pfrc{An^\lambda}{H}^{\frac{13}{84}}\pfrc{n}{A}^{\frac{55}{84}}
    \cdot n^\eps\ll \frac{n}{AT_0} 
    n^{\frac{13}{84}\lambda}T_0^{\frac{13}{84}}\pfrc{n}{A}^{\frac{55}{84}}n^\eps\\
    &\ll n^{3\delta}\cdot n^{\frac{13}{84}\lambda} 
    n^{\frac{1}{3}\cdot\frac{13}{84} - \frac{13}{84}\delta}
    n^{\frac{1}{3}\cdot\frac{55}{84} + \frac{55}{84}\cdot 2\delta} \\
    &\ll n^{\frac{1}{3} - \frac{4}{63} + \frac{349}{84}\delta + \frac{13}{84}\lambda + \eps}.
\end{align*}
The desired result follows.
\end{proof}

\section{Proof of Theorem \ref{lowerbound}}
\label{sec:pf_lower}
In this section, we prove Theorem \ref{lowerbound}, repeated below for the reader's convenience. 

\begin{thm-repeat}
We have the lower bound
\[
    P(n)\gg\frac{\log n}{\log\log n}
\] 
for all sufficiently large $n$.
\end{thm-repeat}

\vspace{1.5mm}

Shallit \cite{shallit} and Erd\H{o}s-Shallit \cite{erdosshallit} established lower bounds for $P(n)$ of $c\frac{\log n}{\log\log n}$ and $c\log n$, respectively, (only) for positive integers $n$ such that $n+1$ is divisible by all sufficiently small positive integers. Such positive integers $n$ will cause the process $x \mapsto n \pmod{x}$ to repeatedly decrement by $1$ at the end. We establish a lower bound that is valid for all positive integers by choosing a starting number based on $n$ that causes the process $x \mapsto n \pmod{x}$ to repeatedly decrement by $1$ at the beginning, for ``Archimedean" reasons rather than ``arithmetic" ones. 

\vs

We will need the following elementary lemma. 

\begin{lemma}\label{elementaryinequality}
There exists $c > 0$ so that the following holds for sufficiently large $n \in \N$. For any $k \in \N$ with $k\le c\frac{\log n}{\log \log n}$, one has
$$(-1)^k k!\left(\sum_{j=0}^k \frac{(-1)^j}{j!}-\frac{1}{e}\right)n > \frac{n}{k+2}+k!.$$
\end{lemma}

\begin{proof}
Note, from the power series for $e^{-1}$, that 
\[
    (-1)^k(k + 2)k!\left(\sum_{j = 0}^k \frac{(-1)^j}{j!} - \frac{1}{e}\right) = 
    1 + \frac{1}{(k + 3)(k + 1)} - O\pfrc{1}{k^3} = 1 + \frac{1}{k^2} - O\pfrc{1}{k^3},
\]
which is greater than $1+\frac{k!(k+2)}{n}$ for sufficiently large $n$, by assumption. 
\end{proof}

\vs

\begin{proof}[Proof of Theorem \ref{lowerbound}]
By adjusting the implied constant, we may assume $n$ is sufficiently large. Let $a = \lfloor (1-\frac{1}{e})n \rfloor$, $a_0 = a$, and $a_{k+1} = n \pmod{a_k}$ for $k \ge 0$. Let $b_0 = (1-\frac{1}{e})n$ and $b_k = (-1)^k k!\left(\sum_{j=0}^k \frac{(-1)^j}{j!}-\frac{1}{e}\right)n$ for $k \ge 1$. 

\vs

We show $P(a,n) \ge c\frac{\log n}{\log\log n}$, where $c > 0$ is as in Lemma 4.3. 

\vs

We prove inductively that $|a_k-b_k| \le k!$ and $a_k = n-ka_{k-1}$. For $k=0$, the first is clearly true. The second is true for $k=1$ and thus so is the first. Now assume they are both true for some $k \ge 1$. We have by Lemma \ref{elementaryinequality} that $a_k \ge b_k-k! > \frac{n}{k+2}$. Since $\lfloor \frac{n}{a_k} \rfloor$ must strictly increase, we have $a_k < \frac{n}{k+1}$. Therefore, $a_{k+1} = n-(k+1)a_k$ and thus $\left|a_{k+1}-b_{k+1}\right| = \left|(n-(k+1)a_k)-(n-(k+1)b_k)\right| = (k+1)\left|a_k-b_k\right| \le (k+1)!$. We have thus shown $a_k > \frac{n}{k+2} > 0$ as long as $k \le c\frac{\log n}{\log \log n}$. 
\end{proof}

\vs

\section{Acknowledgments}

The first author would like to thank his advisor, Ben Green, for encouragement, and Jeffrey Shallit for striving to revive study of this problem. 


\vs

\vs


\begin{thebibliography}{10}

\bibitem{B} J. Bourgain. Decoupling, exponential sums and the Riemann zeta function, J. Am. Math. Soc. 30(1), 205-224 (2017).

\bibitem{erdosshallit} 
P. Erd\H{o}s, J. Shallit. New bounds on the length of finite Pierce and Engel series, S\'em. Th\'eor. Nombres Bordeaux (2) 3, 43–53 (1991).

\bibitem{IK} Iwaniec, H., Kowalski E., Analytic number theory, Amer. Math.
Soc. Colloquium Publ. 53, Amer. Math. Soc., Providence RI, 2004

\bibitem{mrt}
 K. Matomäki, M. Radziwill, T. Tao. Correlations of the von Mangoldt and higher divisor functions I. Long shift ranges. Proc. London Math. Soc. 118, 284–350 (2019).

\bibitem{pierce}
T. A. Pierce. On an Algorithm and Its Use in Approximating Roots of Al- gebraic Equations, Amer, Math. Monthly 36, 523-525 (1929).

\bibitem{shallit}
J. O. Shallit. Metric theory of Pierce expansions, Fibonacci Quart. 24, 22-40 (1986).

\end{thebibliography}
\end{document}